\documentclass[11pt]{amsart}
\usepackage{amsfonts,url}
\usepackage{color}
\usepackage{xcolor}
\usepackage{todonotes}
\usepackage{amssymb, amsmath, amsthm, amscd, accents, mathtools}
\usepackage[T1]{fontenc}
\usepackage{hyperref}
\usepackage{enumitem}
\allowdisplaybreaks
\makeatletter
\let\orgdescriptionlabel\descriptionlabel
\renewcommand*{\descriptionlabel}[1]{%
  \let\orglabel\label
  \let\label\@gobble
  \phantomsection
  \edef\@currentlabel{#1\unskip}%
  \let\label\orglabel
  \orgdescriptionlabel{#1}%
}
\makeatother

\let\oldtocsection=\tocsection

\let\oldtocsubsection=\tocsubsection

\renewcommand{\tocsection}[2]{\hspace{0em}\oldtocsection{#1}{#2}}
\renewcommand{\tocsubsection}[2]{\hspace{1em}\oldtocsubsection{#1}{#2}}

\newtheorem{lemma}{Lemma}
\newtheorem{theorem}[lemma]{Theorem}

\newtheorem{corollary}[lemma]{Corollary}
\theoremstyle{definition}
\newtheorem{definition}[lemma]{Definition}
\newtheorem{remark}[lemma]{Remark}

\numberwithin{lemma}{section}
\numberwithin{equation}{section}

\DeclareMathOperator{\Div}{div}

\newcommand*\mR{\mathbb{R}}
\newcommand*\mL{\mathcal{L}}

\newcommand*\eps{\varepsilon}
\newcommand*\alow{\alpha}

\subjclass[2020]{49L12, 35S10, 35A01, 35K08, 45K05, 35K61, 35B65, 35A09}
\title[Schauder theory for nonlocal viscous HJ equations]{Towards a Schauder theory for fractional viscous Hamilton--Jacobi
equations
}
\author{Espen R. Jakobsen}
\address{ Department of Mathematical Sciences, Norwegian University of Science and Technology, Trondheim, Norway}
\email{espen.jakobsen@ntnu.no}
\author{Artur Rutkowski}
\address{Department of Mathematical Sciences, Norwegian University of Science and Technology, Trondheim, Norway}
\address{Faculty of Pure and Applied Mathematics,
Wroc\l aw University of Science and Technology, Wroc\l aw, Poland}
\email{artur.rutkowski@pwr.edu.pl}

\begin{document}

\maketitle
\begin{abstract}
    We survey some results on Lipschitz and Schauder regularity estimates for viscous Hamilton--Jacobi
    equations with subcritical L\'evy diffusions. 
    The Schauder estimates, along with existence of smooth solutions, are obtained with the help of a Duhamel formula and $L^1$ bounds on the spatial derivatives of the heat kernel. 
    Our results cover very general nonlocal and mixed local-nonlocal diffusions, including strongly anisotropic, nonsymmetric, mixed order, and spectrally one-sided models.
\end{abstract}

\section{Introduction}
The goal of this note is an elementary exposition of some existing results on classical well-posedness and Schauder regularity theory for viscous Hamilton--Jacobi (HJ) equations. We are interested in the following problem:
 \begin{align}\label{eq:HJ}
    \begin{cases}
     \partial_t u - \mL u + H(t,x,u,Du) = f(t,x)\quad &{\rm in}\ (0,T)\times \mR^d,\\
     u(0) = u_0\quad &{\rm in}\ \mR^d.
     \end{cases}
 \end{align}
 Here, $\mL$ is a translation invariant diffusion operator of the form:
\begin{equation}\label{eq:operator}
    \mL u = B\cdot Du + \Div(A Du)+ \int_{\mR^d}(u(\cdot+z) - u(\cdot) - Du(\cdot)\cdot z \textbf{1}_{B(0,1)}(z))\, \nu(dz),
\end{equation}
 with L\'evy triplet $(B,A,\nu)\in \mR^d \times \mR^{d\times d}\times \mathcal M_+(\mR^d)$ where 
 $A\geq0$ 
 and the L\'evy measure $\nu$ is a 
 nonnegative Borel measure satisfying $\int_{\mR^d} (1 \wedge |z|^2)\nu(dz) < \infty$. The function $H\in C((0,T)\times\mR^d\times \mR\times \mR^d)$ is called the Hamiltonian.  Throughout the paper we assume that the operator $\mL$ is subcritical, i.e. a negative operator of order $\alpha\in(1,2]$. The order is $2$ when $A>0$, otherwise it depends on the behavior of $\nu$ near $z=0$. 
The order condition is implicitly formulated through the heat kernel of $\mL$, see Section~\ref{sec:assume} below.  Because the linear diffusion term is the principal term, we call \eqref{eq:HJ} a \textit{viscous} HJ equation. 
\begin{definition}
 We say that $u$ is a classical solution to \eqref{eq:HJ} if $u\in C([0,T]\times \mR^d)$, $\partial_t u, Du, \mL u\in C((0,T)\times \mR^d)$ and the equation and the initial condition are satisfied pointwise.
 \end{definition}
For linear (possibly nonlocal) PDEs in non-divergence form, Schauder regularity theory loosely speaking states that if data and coefficients are $\beta$-H\"older continuous, then solutions are $\alpha+\beta$ H\"older continuous where $\alpha$ is the order of the equation \cite{MR0473443,MR0241822,MR1465184,MR1246036}. In this paper we will give two results of this type for the nonlinear equation \eqref{eq:HJ}. As opposed to many other papers on regularity theory, our method of proof also produces well-posedness of solutions with this regularity.
%
%
 In Section~\ref{sec:regular} we give uniform in time Schauder estimates for smooth initial conditions and Lipschitz $f$, and in Section~\ref{sec:blowup} we have Schauder estimates for initial data and $f$ with low H\"older regularity where gradients will blow-up for small times. For the first type of results, to handle locally Lipschitz Hamiltonians in the long time existence proof, we need uniform Lipschitz estimates on the solutions. Such estimates along with a comparison principle are proved in Section~\ref{sec:Lip}.

 The viscous HJ equations arise in the optimal stochastic control problems or differential games. The unknown function $u$ then represents the (upper or lower) value function  for the optimally controlled process. Equation \eqref{eq:HJ} corresponds to the case of controlled drift but not diffusion. When $H$ is convex in $Du$, it corresponds to the control problem 
 \begin{align*}
     u(T-t,x)=\inf_{\alpha_s\in\mathcal A} E\Big(\int_t^T\big[L(X^{t,x}_s,\alpha_s,s)+f(s,X^{t,x}_s)\big]\,ds+ u_0(X_T^{t,x})\Big), 
 \end{align*}
 where $L$ and $f$ are running costs, $u_0$ a terminal cost, $\mathcal A$ is the set of admissible controls, and $X_s=X_s^{t,x}$ is the solution of an SDE driven by a L\'evy process $N_s$ with generator $\mathcal L$,
$$dX_s=  b(\alpha_s,X_s)\,ds + dN_s,\quad s>t; \qquad X_t=x.$$
 We refer e.g. to Fleming and Soner \cite[Chapter~IV.3]{MR2179357} for more details in the case of Gaussian noise. For non-Gaussian (L\'evy) noise see Øksendal and Sulem \cite{MR3931325}, and for differential games see Carmona \cite{MR3629171}.
 
 Early contributions to viscous HJ equations were given by Kru\v{z}kov \cite{MR0404870} and Crandall and Lions \cite{MR690039}.
Nowadays they have an extensive literature and are still an active area of research. With hundreds of papers, we sample a few to give an idea of some of activity which encompasses well-posedness, regularity, boundary value problems, long time/ergodic/asymptotic problems, singular solutions, approximation and numerics, homogenization, control and large deviations, and stochastic perturbations \cite{MR4000845,MR3285244,MR2732926,MR3575590,MR4264951,MR4333510,MR3095206,MR2664465,MR4044675}. 
The list is very far from being complete or even representative.  Our own main motivation comes from the theory of mean field games with individual noises. There the Nash equilibria are characterized by a forward-backward system consisting of a viscous HJ equation and a Fokker--Planck equation, see \cite{MR4214773,MR3967062,MR3752669} for the background and \cite{MR3934106,MR4309434} for mean field games with nonlocal diffusions.
 
 
 In equation 
 \eqref{eq:HJ} 
 the diffusion term dominates and induces a regularizing effect, which allows to solve the equation with the help of the heat kernel of $\mL$ and a Duhamel formula, yielding the so-called mild solutions. In the nonlocal case, this approach to the viscous HJ equation was used by Imbert \cite{MR2121115} and many other authors \cite{MR2471928,MR2259335,MR4309434,MR3623641,ERJAR,MR2373320}, whereas the local case goes back at least to Ben-Artzi \cite{MR1186789}. We note that for regular data the Duhamel formula for the viscous HJ equation works somewhat similar to the linear case, at least for the short-time results. In this connection we mention the work of Mikulevi\v{c}ius and Pragarauskas \cite{MR1246036} who study regularization properties for the heat equation with quite general L\'evy-type diffusions. More recently, Schauder estimates (without existence of classical solutions) were developed by Dong, Jin and Zhang \cite{MR3803717} for fully nonlinear equations involving nonlocal operators with the L\'evy measure comparable to the one of the fractional Laplacian. The results given below cover a large class of translation invariant diffusions which were not present in \cite{MR3803717}, including strongly anisotropic and nonsymmetric stable operators.

The long-time existence under \textit{local} regularity assumptions for $H$ is more delicate, as it requires global Lipschitz bounds for the solutions, which do not follow from the Duhamel formula. In general, these Lipschitz bounds can be obtained using the comparison principle and viscosity solution techniques, see Imbert \cite{MR2121115} and Jakobsen and Karlsen \cite{MR2129093}. If the solution is known to be more regular pointwise in time, one can also apply the Bernstein method, which is quite straightforward when one assumes that $f(t) \in W^{1,\infty}(\mR^d)$ uniformly in $t$. A more sophisticated Bernstein-type argument can be applied in order to get bounds for $u$ in terms of weaker norms of $f$, see e.g. Goffi \cite{Goffi}, but this is beyond the scope of this note.

In case that $u_0,f(t) \in C^\beta_b(\mR^d)$ for $t\in (0,T)$ with $\beta \in (0,1)$ we can still obtain classical solutions, but $Du(t)$ and $\mL u(t)$ can, in general, blow up for small $t$. There are many works concerning the viscous HJ equations with irregular initial data, see e.g., Gilding, Guedda, and Kersner \cite{MR1998665}, Benachour and Lauren\c{c}ot \cite{MR1720778} for the local diffusions and \cite{MR4309434,MR2121115,ERJAR,MR2737806} for the nonlocal case. However, we are not aware of papers which explicitly study blow up rates for $u_0 \in C^\beta_b(\mR^d)$ data in the vein of Theorem~\ref{th:blowup} below, obtained by the authors together with Amund Bergset, cf. his master thesis \cite{Amund}. Admittedly, we only allow Hamiltonians with global Lipschitz regularity there so far; an extension to locally Lipschitz Hamiltonians seems non-trivial and interesting.
In a future work we aim to study mean field game systems with H\"older payoff functions. The blowup for the viscous HJ equation means that the drift term in the corresponding Fokker--Planck equation may become large close to the terminal time, which poses a difficulty for compactness arguments.

 \section{Assumptions}\label{sec:assume}
We adopt rather standard assumptions on $H$ and $f$, which we will use in Sections~\ref{sec:Lip} and \ref{sec:regular}. We note that by assuming regularity for higher order derivatives in \ref{eq:H} and \ref{eq:F0} below, one could get higher regularity of solutions to the viscous HJ equation. Here, for the sake of clarity, we work with quite weak assumptions, but strong enough to allow for the classical solution setting.
\begin{description}
\item[(H0)\label{eq:H0}]$H\colon [0,\infty)\times \mR^d\times\mR\times \mR^d\to \mR$ is continuous and for every $R$ there exists $C_R$ such that for every $t, |u|\leq R$, $x,y\in \mR^d$ and $p,q \in B_R$,
\begin{align*}
    |H(t,x,u,p) - H(t,y,u,q)| \leq C_R(|x-y| + |p-q|).
\end{align*}
    \item[(H1)\label{eq:H}]$H$ is smooth and for all $l\in \mathbb{N}^{2d+1}$ with $|l|\leq 2$, $\sup\limits_{x\in \mR^d}|D^l_{x,u,p}H(\cdot,x,\cdot,\cdot)|$ is locally bounded.
\item[(H2)\label{eq:H2}] For every $R>0$ there exists $C_R>0$ such that for $x,y\in \mR^d$, $t,|u|\leq R$, and $p\in \mR^d$,
$$|H(t,x,u,p) - H(t,y,u,p)|\leq C_R(1+|p|)|x-y|.$$
\item[(H3)\label{eq:H3}] For every $T>0$ there exists $\gamma = \gamma(T)\in \mR$ such that for all $t\in [0,T]$, $x\in \mR^d$, $u\leq v$, and $p\in\mR^d$,
$$H(t,x,v,p) - H(t,x,u,p) \geq \gamma (v-u).$$
\item[(F0)\label{eq:F0}] $f\in C_b([0,T]\times \mR^d)$ and for every $T>0$ there exists $C$ such that for $t\in [0,T]$ and $x,y\in \mR^d$
\begin{align*}
    |f(t,x) - f(t,y)|\leq C|x-y|.
\end{align*}
\end{description}
Note that \ref{eq:H} is strictly stronger than \ref{eq:H0}.

In the Schauder estimates we use the fact that the order of the operator is greater than 1. We express that through appropriate estimates for the heat kernel. If we let $\psi$ be the Fourier multiplier corresponding to $\mL$ (see e.g. \cite{MR3587832}), then the heat kernel of $\mL$ is defined as 
\begin{align*}
    K_0 = \delta_0,\quad \mathcal{F}(K_t)(\xi) = e^{-t\psi(\xi)},\quad t>0.
\end{align*}
\begin{description}
\item[(K)\label{eq:K}]
        There is $\mathcal K > 0$ and $\alow\in (1,2]$, such that the heat kernels $K$ and $K^*$ of $\mL$ and $\mL^*$ respectively are smooth densities of probability measures, and for $\tilde K = K,K^*$ and $\beta \in \mathbb{N}^d$ we have
        $$\|D^{\beta}\tilde K(t,\cdot)\|_{L^1(\mR^d)} \leq \mathcal Kt^{-\frac {\beta}{\alow}}.$$   
\end{description}
Under \ref{eq:K} we have $K_t\in C_0^\infty(\mR^d)$ for every $t>0$ and $K_t$ satisfies $(\partial_t - \mL)K_t = 0$ pointwise on $(0,\infty)\times \mR^d$.
Here are some examples of operators satisfying \ref{eq:K} (taken mostly from \cite[Section~4]{MR4309434}): 
\begin{enumerate}
    \item[(i)] There is $c>0$ such $\langle Ax,x\rangle \geq c|x|^2$ for $x\in \mR^d$, e.g. $\mL = \Delta$. Then \ref{eq:K} is satisfied with $\alow=2$.\smallskip
    \item[(ii)] The L\'evy measure $\nu$ has an absolutely continuous part with density 
    $\nu_{ac}(z) \approx |z|^{-d-\alow}$ for all $|z|\leq 1$ and $\alow\in (1,2)$, see \cite[Theorem~4.3]{MR4309434}, based on the pointwise estimates in \cite{MR4308627}.\smallskip
    \item[(iii)] $\mL  = -(-\partial^2_{x_1})^{\alpha_1/2}-(-\partial^2_{x_2})^{\alpha_2/2} -\ldots -(-\partial^2_{x_d})^{\alpha_d/2}$, where $\alpha_i\in (1,2]$. Then $\alow = \min_i \alpha_i$.\smallskip
    \item[(iv)] The family of Riesz--Feller operators on $\mR$, including the spectrally one-sided model with $A=0$ and $\nu(z) = |z|^{-1-\alow}\textbf{1}_{(0,\infty)}(z)$, see \cite[Lemma~2.1 (G7) and Proposition~2.3]{MR3360395}.\smallskip
    \item[(v)] The generator of the CGMY model in finance, see \cite[Example~4.4]{MR4309434}.\smallskip
    \item[(vi)] If $\mL$ satisfies \ref{eq:K} and $L$ is any other L\'evy operator, then $\mL + L$ satisfies \ref{eq:K}.   
\end{enumerate}

\section{Comparison principle and Lipschitz estimates}\label{sec:Lip}
In the next sections we will construct short-time solutions via fixed point arguments and we will get rather strong regularity estimates, but only on short time intervals. For the sake of obtaining long-time solutions we will establish a global in time Lipschitz bound using different methods in Theorem~\ref{th:Lip} below.

We first give a comparison principle in order to obtain supremum bounds. A similar result was given by Imbert \cite[Theorem~2]{MR2121115} for purely nonlocal $\mL$ in the (more general) viscosity solution setting; it extends rather easily to general L\'evy operators.
\begin{theorem}\label{th:max}
    Assume that \ref{eq:H0}, \ref{eq:H2}, \ref{eq:H3}, and \ref{eq:F0} hold. Let $u_0$ be uniformly continuous and bounded and suppose that $u$ and $v$ are respectively classical sub- and supersolution to \eqref{eq:HJ}, i.e., $u(0,x)\leq u_0(x)\leq v(0,x)$ for $x\in \mR^d$ and
    \begin{align*}
     &\partial_t u - \mL u +H(t,x,u,Du) \leq f(t,x)\quad {\rm in}\ (0,T)\times \mR^d,\\
    &\partial_t v - \mL v + H(t,x,v,Dv) \geq f(t,x)\quad {\rm in}\ (0,T)\times \mR^d.
 \end{align*}
 Then $u\leq v$ on $[0,T)\times \mR^d$.
\end{theorem}
\begin{proof}
    Note that under present assumptions $f$ can be absorbed into $H$, so we can assume that $f=0$. Since the diffusion operator $\mL$ is linear and translation invariant, it suffices to follow the proof of Imbert. First of all, since the L\'evy measure $\nu$ is finite outside neighborhoods of 0, there exists a function $\varphi\in C^2(\mR^d)$ such that
    \begin{align*}
        \|D\varphi\|_{\infty} + \|D^2\varphi\|_{\infty} \leq C,\quad \lim\limits_{|x|\to\infty} \varphi(x) = \infty,\quad \mL\varphi\in C(\mR^d).
    \end{align*}
    Since the sub/supersolutions are classical, we do not need to introduce the sub- and superdifferentials after \cite[(14)]{MR2121115}, because the operator can be evaluated pointwise, and we have (in the language of \cite{MR2121115}) $DU(\overline{t},\overline{x}) - DV(\overline{s},\overline{y}) = \alpha D\varphi(\overline{x})$. Then it is easy to obtain the following counterpart of inequality \cite[(16)]{MR2121115}:
    \begin{align*}
        \mL V(\overline{s},\overline{y}) - \mL U(\overline{t},\overline{x}) \leq -\alpha \mL \varphi(\overline{x}).
    \end{align*}
    The rest of the proof does not depend on the form of the diffusion operator.
\end{proof}
\begin{corollary}\label{cor:supbound}
Assume that \ref{eq:H0}, \ref{eq:H2}, \ref{eq:H3}, and \ref{eq:F0} hold. If $u_0$ is uniformly continuous and bounded and $u$ is a classical solution to \eqref{eq:HJ}, then 
\begin{align}\label{eq:supbound}
    \|u(t)\|_{\infty} \leq e^{-\gamma t}\|u_0\|_{\infty}+(1-\tfrac 1\gamma e^{-\gamma t})(\|f\|_{\infty} + \|H(\cdot,\cdot,0,0)\|_{\infty}),\quad t\in[0,T).
\end{align}
\end{corollary}
\begin{proof}
    Let $w(t)$ denote the right hand side of \eqref{eq:supbound}. Then the proof follows by comparison since by \ref{eq:H3} $w$ and $-w$ are super- and subsolutions of \eqref{eq:HJ}. 
\end{proof}
\begin{theorem}\label{th:Lip}
Assume that \ref{eq:H0}, \ref{eq:H2}, \ref{eq:H3}, and \ref{eq:F0} hold, $u_0\in C^1_b(\mR^d)$ and let $u$ be a classical solution to \eqref{eq:HJ}. Then
\begin{align*}
\|Du\|_{\infty} \leq C(T,H,f)(1 + \|Du_0\|_{\infty}),\quad t\in [0,T).
\end{align*}
\end{theorem}
\begin{remark}
We only give the proof in the case of regular data and solutions to showcase the elegant and more accessible Bernstein method. For less regular data one can apply the viscosity solution proof of Imbert \cite{MR2121115}. In order to adapt the proof of Imbert to mixed local-nonlocal operators an appropriate  viscosity solution formulation is needed, see e.g., Jakobsen and Karlsen \cite{MR2129093} in this regard.
\end{remark}
\begin{proof}[Proof of Theorem~\ref{th:Lip} for regular data and solutions] 

If \ref{eq:H} is satisfied, $f\in C([0,T],C^2_b(\mR^d))$, and $u\in C([0,T],C^3_b(\mR^d))$, then we can differentiate the equation in $x_i$ and get
\begin{equation}\label{eq:diffHJ}
\begin{split}
    \partial_t u_{x_i} - \mL u_{x_i} + H_{x_i}(t,x,u,Du) &+ u_{x_i} H_u(t,x,u,Du)\\ &+ Du_{x_i} \cdot H_p(t,x,u,Du) = f_{x_i}.
\end{split}
\end{equation}
Now, let $w = \frac 12|Du|^2$. If we multiply \eqref{eq:diffHJ} by $u_{x_i}$ and sum over $i$ we get
\begin{align*}
    \partial_t w - \mL w + \sum\limits_i \Gamma(u_{x_i})  &+ Du \cdot H_x(t,x,u,Du) + 2w H_u(t,x,u,Du)\\ &+ Dw \cdot H_p(t,x,u,Du) = Df\cdot Du,
\end{align*}
where (recall that $A\geq 0$)
\begin{align*}
\Gamma(u) = \tfrac 12(\mL u^2 - 2u \mL u) = \tfrac 12\int_{\mR^d} (u(x+y) - u(x))^2\, \nu(dy) + Du \cdot A Du\geq 0.
\end{align*}
Furthermore, we have $w\geq 0$ and $H_u\geq \gamma$ by \ref{eq:H3}. Therefore, by \ref{eq:H2} and Corollary~\ref{cor:supbound},
\begin{align*}
    &\partial_t w - \mL w  +  Dw \cdot H_p(t,x,u,Du)\\ \leq&\,  -Du \cdot H_x(t,x,u,Du) -2\gamma w + \tfrac 12\|Df\|_{\infty}^2 + \tfrac 12 w\\
    \leq&\, C(T,H,f)(1 + w).
\end{align*}
Let $C = C(T,H,f)$ and note that $v = e^{Ct}(1+\|w(0)\|_{\infty}) -1$ satisfies
\begin{align*}
    \partial_t v - \mL v + Dv \cdot H_p(t,x,u,Du) = C(1+v),\quad v(0) = \|w(0)\|_{\infty} \geq w(0).
\end{align*} Therefore by the comparison principle for the linear parabolic problems we get that $w \leq v$, which ends the proof in this case.
\end{proof}
\section{Uniform Schauder estimates with regular initial data} \label{sec:regular}
In this section we give uniform in time Schauder estimates under the assumption that the initial data is regular. The result and proof were essentially given in \cite{ERJAR}. We repeat the arguments here in order to demonstrate the technique where the regularity of the heat kernel is used to gain $\alow-\varepsilon$ derivatives over the forcing term $f$. Similar results were given earlier by, e.g., Imbert \cite{MR2121115} and Ersland and Jakobsen \cite{MR4309434}, but with only one derivative gain. It seems that in order to gain full $\alpha$ regularity one needs a stronger assumption than \ref{eq:K}, see e.g. \cite{MR3803717} or \cite{MR4056997}.



In the proof we first show the existence of a short-time mild solution using the Duhamel formula and the Banach fixed point theorem. Then we show that the mild solution is classical, which enables us to use the results of the previous section to obtain the long time existence. Finally, we get the uniqueness from the comparison principle. 
\begin{theorem}\label{thm:main1}
    Let $\eps \in (0,\alpha-1)$ and assume that \ref{eq:H}, \ref{eq:H2}, \ref{eq:H3}, \ref{eq:F0}, and \ref{eq:K} are satisfied and let $u_0 \in C^{1+\alpha - \eps}_b(\mR^d)$. Then problem \eqref{eq:HJ} has a unique classical solution $u$ satisfying
\begin{align}\label{eq:ubounds}
&\|\partial_t u\|_{\infty} + \sup\limits_{t\in[0,T]} \|u(t,\cdot)\|_{C^{1+\alow - \eps}_b(\mR^d)}\\[0.1cm] &\leq \, c(T,\mathcal{K})\Big(c(H) + \sup\limits_{t\in[0,T]} \|f(t,\cdot)\|_{W^{1,\infty}(\mR^d)}+\|u_0\|_{C^{1+\alpha-\eps}_b(\mR^d)}\Big).\nonumber
\end{align}
\end{theorem}
\begin{proof}
    We denote $\sigma = \alpha - 1 - 
    \eps$, so that $1+\alpha - \eps = 2+\sigma$.
\smallskip

\noindent    \textbf{1) Short-time mild solutions by Banach's fixed point theorem.}
   Let 
  \begin{align*}
   R = R(u)= \sup_t \|u(t)\|_{C^1_b(\mR^d)}, \quad\!   R_1 = 
   2(\|u_0\|_{C^{2+\sigma}_b} + \sup_t\|f(t)\|_{W^{1,\infty}(\mR^d)}) + 1,
  \end{align*}
  and define the space
  \begin{align*}X = \{u\in C_b((0,T],C^{2+\sigma-\eps}_b(\mR^d)): \sup\limits_{t\in[0,T]}\|u(t,\cdot)\|_{C^{2+\sigma}_b(\mR^d)} \leq R_1\}.
  \end{align*}
  Note that for all $u\in X$ we have $R\leq R_1$.
  For $u\in X$, $t\in [0,T)$, $x\in \mR^d$ we define the fixed-point map:
  \begin{align*}
      S(u)(t,x) := K_t \ast u_0(x) + \int_0^t K_{t-s}\ast [H(s,\cdot,u,Du) + f(s)](x) \, ds.
  \end{align*}
  We will now estimate the $C^{2+\sigma}_b$ norms of $S(u)$ in order to see that it belongs to $X$. The time continuity of the $C^{2+\sigma-\eps}_b$ norms follows from the $C^{2+\sigma}_b$ estimates, see e.g. \cite[Lemma~B.2]{ERJAR}. 
  First note that for $k=1,2$ we have
  \begin{equation}\label{eq:dks}\begin{split}
      &D_x^k S(u)(t,x)\\ =\, &K_t \ast D_x^k u_0(x) + \int_0^t D_xK_{t-s}\ast D_x^{k-1}[H(s,\cdot,u,Du) + f(s)](x) \, ds.
      \end{split}
  \end{equation}
  By \ref{eq:K} and interpolation for H\"older functions (e.g. \cite[Remark~2.22 d)]{ERJAR}),
   
  \begin{align*}
  &\int_0^T\|D_xK_{t-s}\|_{L^1(\mR^d)}\, ds\leq \mathcal{K}\int_0^T(t-s)^{-\frac 1\alpha}\, ds=\frac{\alpha-1}{\alpha}\mathcal{K}T^{\frac{\alpha-1}{\alpha}}:=c_{T,\alpha},\\[0.3cm]
      &\int_0^T\frac{\|D_xK_{t-s}(\cdot+h)- D_xK_{t-s}(\cdot)\|_{L^1(\mR^d)}}{|h|^{\sigma}}\, ds \leq \mathcal{K}\!\int_0^T (t-s)^{-\frac{1+\sigma}{\alpha}}\, ds =: c_{T,\alpha,\sigma},
  \end{align*}
  for all $h\in \mR^d\setminus \{0\}$. Therefore by Young's inequality for convolutions and
   \ref{eq:H} we have for $k=0,1,2$ (for $k=0$ we use $\|K_{t-s}\|_{L^1(\mR^d)} = 1$): 
  \begin{align*}
      \|D^kSu(t)\|_{\infty} \leq&\,  \|D^ku_0\|_{\infty} + (T\vee c_{T,\alpha})\Big(\sup\limits_t\|f(t)\|_{W^{1,\infty}(\mR^d)}\\
      &\qquad \qquad\quad  + C(H,R)(1+\sup_t\|u(t,\cdot)\|_{C^{2}_b(\mR^d)}) \Big),\\[0.2cm]
      [D_x^2Su(t)]_{C^{\sigma}(\mR^d)} \leq& \,[D_x^2u_0]_{C^{\sigma}(\mR^d)} +c_{T,\alpha,\sigma}\Big(\sup_t\|f(t)\|_{W^{1,\infty}(\mR^d)}\\ &\qquad\qquad\quad + C(H,R)(1+\sup_t\|u(t,\cdot)\|_{C^{2}_b(\mR^d)})\Big).
  \end{align*}
  Since $c_{T,\alpha}$ and $c_{T,\alpha,\sigma}$ converge to 0 as $T\to 0^+$, by taking $T$ small enough we get $\sup_t\|Su(t)\|_{C^{2+\sigma}_b(\mR^d)} \leq R_1$. 
  Thus $S$ maps $X$ to $X$.
  
  Then it remains to show that $S$ is a contraction. The computations are similar to the above, we only inspect the highest order term to see that the bounds in \ref{eq:H} are sufficient: by \eqref{eq:dks} we have
  \begin{align*}
      &[D_x^2(S(u) - S(v))]_{C^{\sigma}(\mR^d)}\\ &\leq \mathcal{K}\int_0^t (t-s)^{-\frac{\sigma + 1}{\alpha}} \|D_x (H(s,\cdot,u,Du) - H(s,\cdot,v,Dv))\|_{\infty}\, ds\\ &\leq c_{T,\alpha,\sigma}
      C(H,R)\|u - v\|_{C^2_b(\mR^d)}.
  \end{align*}
  Again, by taking $T$ small enough we get that $S$ is a contraction on $X$. Therefore, by the Banach fixed point theorem $S$ has a unique fixed point on $X$, which we call the mild solution.
\medskip

\noindent \textbf{2) Mild solutions are classical solutions.}\  The quickest way to prove that is by using the Duhamel formula on $(t,t+h)$ with $h>0$:
  \begin{align*}
      u(t+h,x) = K_{h}\ast u(t)(x) + \int_t^{t+h} 
      K_{t+h-s} \ast[H(s,\cdot,u,Du) + f(s)](x)\, ds.
  \end{align*}
  Hence,
  \begin{align*}
      \frac{u(t+h,x) - u(t,x)}{h} =& \frac{K_{h}\ast u(t)(x)-u(t,x)}{h}\\ &+ \frac 1h\int_t^{t+h} K_{t+h-s} \ast[H(s,\cdot,u,Du) + f(s)](x)\, ds.
  \end{align*}
  As $h\to 0^+$, the first term on the right-hand side converges to the generator of the semigroup associated with $K_{t}$, but since $u\in C^2_b(\mR^d)$ this is equal to $\mL u(t,x)$, see e.g. Sato \cite[Theorem~31.5]{MR3185174}.\footnote{The result is given for $C^2_0(\mR^d)$, but an easy cut-off argument suffices.} The second term converges to $H(t,x,u,Du) + f(t,x)$ by the regularity of $f,$ $H$ and $u$, and the fact that $K_t\to \delta_0$ weakly as $t\to 0^+$. Since the right-hand side converges, the left-hand side converges too, and is equal to $\partial_t u$, which proves that $u$ is a classical solution.
\medskip

\noindent  \textbf{3) Long-time existence by gluing.}\  The idea for extending the solution in time consists in repeating the fixed-point procedure starting from a later point in time, e.g., $T/2$ and then iterating. However, we need to find appropriate lower bounds for lengths of the short-time existence intervals. To this end we use the fact that the mild solutions are classical and apply Theorem~\ref{th:Lip}. This gives us a locally uniform in time bound for the constant $R$ which determines the length of the existence interval, letting us extend the solution to an arbitrarily long time interval.
\medskip

\noindent  \textbf{4) Uniqueness} follows from the comparison principle in Theorem~\ref{th:max}.
\end{proof}

\section{Schauder estimates with irregular initial data}\label{sec:blowup}
 In an ongoing project with Amund Bergset  we study viscous HJ equations with initial data $u_0 \in C^\beta_b(\mR^d)$ for $\beta \in (0,1)$. 
 In this case $\|Du(t)\|_{\infty}$ blows up as $t\to 0^+$, and therefore we need more explicit and global control of growth/regularity of the Hamiltonian in $p$. For simplicity we consider the following global Lipschitz condition:
\begin{align}\label{eq:HLip}\tag{\textbf{HLip}}
    H = H(p) \ \textnormal{ and }\ H,D_pH\ \textnormal{\ are globally Lipschitz in}\ p.
\end{align}
It should be possible to consider e.g. power-type Hamiltonians, but the analysis would be more delicate in this case. For local diffusions with subquadratic Hamiltonians and measure initial data see \cite{MR1720778}. Superquadratic Hamiltonians were studied e.g. in \cite{MR4104825}.

Our result shows that the $\beta$-H\"older continuity of the data is preserved by the solution, while higher order H\"older norms blow up with a specific rate as $t\to 0$. There are two cases depending on whether $\alpha+\beta$ is greater or smaller than $2$.
\begin{theorem}\label{th:blowup}
Assume that $T>0$, \eqref{eq:HLip} holds, $f\in L^\infty([0,T],C^\beta_b(\mR^d))$, and $u_0\in C^\beta_b(\mR^d)$. Then \eqref{eq:HJ} has a unique classical solution $u$. Moreover for every small $\varepsilon>0$ there exists $C$ such that for all $t\in(0,T]$:
\begin{align*}
        &(i) \ \quad \|u(t)\|_{\infty} + [u(t)]_{C^\beta(\mR^d)} + t^{\frac {1-\beta}{\alow}} \|Du(t)\|_{\infty} \\[0.2cm] 
        &\hspace{2cm}+ t^{\frac{\alow - \varepsilon}{\alow}}[Du(t)]_{\alow+\beta-\varepsilon} \leq C,
         & \text{if $\alow+\beta \leq 2$,}\\[0.4cm]
        &(ii)\quad \|u(t)\|_{\infty} + [u(t)]_{C^\beta(\mR^d)} + t^{\frac {1-\beta}{\alow}} \|Du(t)\|_{\infty}\\[0.2cm] &\hspace{2cm}+ t^{\frac{2-\beta}{\alow}}\|D^2u(t)\|_{\infty} + t^{\frac{\alow - \varepsilon}{\alow}}[D^2u(t)]_{\alow+\beta-\varepsilon} \leq C, & \text{if $\alow+\beta > 2$.}
    \end{align*}

\end{theorem}
The proof uses a similar strategy as the proof of Theorem \ref{thm:main1} above, we refer to \cite{Amund} for details.
Compared to earlier results, this result cover a larger class of nonlocal operators, gives the behaviour of solution smoothness near $t=0$, and provide existence of solutions with the stated smoothness. Our ultimate goal is to use this result to study new classes of Mean Field Games.

\section*{Acknowledgements}
E. R. Jakobsen received funding from the Research Council of Norway under Grant Agreement No. 325114 “IMod. Partial differential equations, statistics and data: An interdisciplinary approach to data-based modelling”. While this project was carried out A. Rutkowski was an ERCIM Alain Bensoussan fellow at NTNU.


\bibliographystyle{abbrv}
\bibliography{bib-file}
\end{document}